\documentclass[leqno]{amsart}
\usepackage{amsmath}
\usepackage{amssymb}
\usepackage{amsthm}
\usepackage{enumerate}
\usepackage[mathscr]{eucal}
\usepackage{mathrsfs}
\usepackage{xcolor}
\theoremstyle{plain}
\usepackage{tikz}
\newtheorem{theorem}{Theorem}[section]
\newtheorem{lemma}[theorem]{Lemma}

\theoremstyle{definition}

\newtheorem{remark}[theorem]{Remark}

\newtheorem{cor}[theorem]{Corollary}
\theoremstyle{remark}

\begin{document}

\title [Extreme points  of $B_{\mathcal{L}(X)_w^*}$ and best approximation in $\mathcal{L}(X)_w$]{Extreme points of the unit ball of $\mathcal{L}(X)_w^*$ and best approximation in $\mathcal{L}(X)_w$}

\author[Arpita Mal]{Arpita Mal}

\email{arpitamalju@gmail.com}






\date{}
\maketitle
\begin{abstract}
We study the geometry of $\mathcal{L}(X)_w,$ the space of all bounded linear operators on a Banach space $X,$ endowed with the numerical radius norm, whenever the numerical radius defines a norm. We obtain the form of the extreme points  of the unit ball of the dual space of $\mathcal{L}(X)_w.$ Using this structure, we explore Birkhoff-James orthogonality, best approximation and deduce distance formula in $\mathcal{L}(X)_w.$ A special attention is given to the case of operators satisfying a notion of smoothness.   Finally, we obtain an equivalence between Birkhoff-James orthogonality in $\mathcal{L}(X)_w$ and that in $X.$
\end{abstract}

\section{Introduction}
The purpose of this article is to explore best approximation in those operator spaces where numerical radius defines a norm using the form of the extreme points of the dual of the operator space. Let us first introduce the notations and terminologies that will be used throughout the article.\\
Let $H$ be a Hilbert space over the field $\mathbb{F},$ where $\mathbb{F}=\mathbb{R}$ or $\mathbb{C}.$  For a bounded linear operator $T$ on $H,$ the numerical radius of $T,$ denoted as $w(T),$ is defined as 
\[w(T)=\sup\{|\langle Tx, x\rangle |:x\in H,\|x\|=1\}.\] 
The natural generalization of this notion when $T$ is defined on a Banach space $X$ over the field $\mathbb{F},$ is as follows:
 \[w(T)=\sup\{|x^*(Tx)|:x^*\in S_{X^*},x\in S_X,x^*(x)=1\},\]
 where $X^*$ is the dual space of $X$ and $S_X$ denotes the unit sphere of the corresponding space $X.$ Suppose that $\mathcal{L}(X)$ (respectively $\mathcal{L}(H)$) is the space of all bounded linear operators on $X$ (respectively $H$). Clearly, the numerical radius $w(.)$  defines a seminorm on $\mathcal{L}(X).$ Very often $w(.)$ actually defines a norm on $\mathcal{L}(X),$ which is equivalent to the operator norm. In this article, we only consider those Banach spaces $X,$ where $w(.)$ is a norm on $\mathcal{L}(X).$    
We denote by $\mathcal{L}(X)_w,$ the space of all bounded linear operators on $X$ endowed with the numerical radius norm.\\
For $x\in X$ and a subspace $Z$ of $X,$ the distance of $x$ from $Z$ is defined as $d(x,Z)=\inf\{\|x-z\|:z\in Z\}.$ An element $z_0\in Z$ is said to be a best approximation to $x$ out of $Z,$ if $\|x-z_0\|=d(x,Z).$ Let $\mathscr{L}_Z(x)=\{z_0\in Z:\|x-z_0\|=d(x,Z)\}.$ In general, $\mathscr{L}_Z(x)$ may be empty. However, if $Z$ is finite-dimensional, then $\mathscr{L}_Z(x)\neq \emptyset$ for all $x\in X.$ Birkhoff-James (B-J) orthogonality and best approximation are closely related notions. For $x,y\in X,$ $x$ is said to be B-J orthogonal \cite{B,J} to $y,$ denoted as $x\perp_B y,$ if $\|x+\lambda y\|\geq \|x\|$ for all scalars $\lambda.$   We say that $x\perp_B Z$ if $x\perp_B z$ for all $z\in Z.$ Observe that $z_0\in \mathscr{L}_Z(x)$ if and only if $x-z_0\perp_B Z.$ For $(0\neq)x\in X,$ suppose that $J(x)=\{f\in S_{X^*}:f(x)=\|x\|\}.$ Note that, $J(x)$ is a non-empty, convex, weak*compact, extremal subset of the unit ball of $X^*.$ If $J(x)$ is singleton, then $x$ is said to be smooth in $X.$ From a well-known result (\cite[Th. 2.1]{J}) of James, it follows that $x\perp_B y$ if and only if there exists $f\in J(x)$ such that $f(y)=0.$ There is another fundamental characterization of B-J orthogonality due to Singer in terms of the extreme points of the unit ball  of $X^*.$ We denote the unit ball of $X$ by $B_X$ and  the set of all extreme points of $B_X$ by $E_X.$ 
\begin{theorem}\cite[Th. 1.1, pp. 170]{S}\label{th-singer}
	Let $\mathbb{X}$ be a Banach space. Let $x\in \mathbb{X},$ $\mathbb{\mathcal{W}}$ be a subspace of $\mathbb{X}$ such that $\dim(\mathcal{W})=n$ and $x\notin \mathcal{W}.$ Then $y\in \mathscr{L}_{\mathcal{W}}(x)$ if and only if there exist
	$h$ extreme points $f_1,f_2,\ldots,f_h\in E_{\mathbb{X}^*},$ ($h\leq n+1$ if $\mathbb{F}=\mathbb{R}$ and $h\leq 2n+1$ if $\mathbb{F}=\mathbb{C}$), $h$ numbers $t_1,t_2,\ldots,t_h>0$ such that $\sum_{i=1}^h t_i=1,$ $\sum_{i=1}^h t_if_i(w)=0$ for all $w\in \mathcal{W}$ and $f_i(x-y)=\|x-y\|$ for all $1\leq i\leq h.$
\end{theorem}
B-J orthogonality has been extensively studied on $\mathcal{L}(X)$ with respect to operator norm (see \cite{BS,G,MP,RAO21,SMP,SPM2} for some references). Since numerical radius is an important concept related to an operator, it is natural and interesting to study B-J orthogonality in $\mathcal{L}(X)_w.$  It is well-known that if $H$ is a complex Hilbert space, then $w(.)$ always defines a norm on $\mathcal{L}(H).$ In \cite{MPS}, we initiated the study of B-J orthogonality on $\mathcal{L}(H)_w.$ Recently, in \cite{RS21}, Roy and Sain studied B-J orthogonality on $\mathcal{L}(X)_w.$ For $T,A\in \mathcal{L}(X)_w,$ $T$ is B-J orthogonal to $A$ with respect to numerical radius norm if $w(T+\lambda A)\geq w(T)$ for all scalars $\lambda.$ We denote it as $T\perp_wA,$ i.e.,
\[T\perp_wA\Leftrightarrow w(T+\lambda A)\geq w(T)~\mbox{ for all scalars }\lambda. \]
Clearly, using the characterization of James, we get $T\perp_wA$ if and only if there exists $f\in J_w(T)$ such that $f(A)=0,$ where
\[J_w(T)=\{f\in \mathcal{L}(X)_w^*:\|f\|=1,f(T)=w(T)\}.\]
  Following \cite{RS21}, we say that $T$ is nu-smooth if and only if $J_w(T)$ is singleton.
 Let $\mathcal{V}$ be a subspace of $\mathcal{L}(X)_w.$ Then we say that $T\perp_w \mathcal{V}$ if $T\perp_w A$ for all $A\in \mathcal{V}.$ In $\mathcal{L}(X)_w,$ we denote the distance of $T$ from $\mathcal{V}$ by $d_w(T,\mathcal{V}),$ i.e., 
$$d_w(T,\mathcal{V})=\inf\{w(T-S):S\in \mathcal{V}\}$$ and the collection of all best approximations to $T$ out of $\mathcal{V}$ by $\mathscr{L}_{\mathcal{V}}(T)_w,$ i.e.,  $$\mathscr{L}_{\mathcal{V}}(T)_w=\{S\in \mathcal{V}:w(T-S)=d_w(T,\mathcal{V})\}.$$ 
For $\lambda\in \mathbb{F}, \Re(\lambda)$ and $\Im(\lambda)$ are respectively the real and imaginary parts of $\lambda.$ If $A\subseteq X,$ then $co(A)$ denotes the convex hull of $A.$ For a set $A\subseteq X$ and $B\subseteq X^*,$ the polar of $A$ denoted as $A^o$ and prepolar of $B$ denoted as ${}^oB$ are defined as follows:
$$A^o=\{x^*\in X^*:|x^*(a)|\leq 1 ~\forall~a\in A\}~\text{and}$$ \[{}^oB=\{x\in X:|b^*(x)|\leq 1 ~\forall~b^*\in B\}.\]  
For $x\in X,x^*\in X^*,$ $x^*\otimes x$ denotes a linear functional on $\mathcal{L}(X)_w$ defined as $x^*\otimes x(T)=x^*(Tx)$ for all  $T\in\mathcal{L}(X)_w.$ Consider the set 
\begin{equation}
	\mathcal{A}=\{x^*\otimes x\in \mathcal{L}(X)_w^*:x^*\in B_{X^*},x\in B_{X},|x^*(x)|=\|x^*\|\|x\|\}.
\end{equation}
Then it is straightforward to check that
\[w(T)=\sup\{|x^*(Tx)|:x^*\otimes x\in\mathcal{A}\}.\]
After this introductory part, the article contains two sections. In section 2, we explore the extreme points of $B_{\mathcal{L}(X)_w^*}.$ We prove that, if $X$ is finite-dimensional, then each extreme point of $B_{\mathcal{L}(X)_w^*}$ assumes a nice form. In particular, it is of the form $x^*\otimes x$ for some $x^*\in E_{X^*}$ and $ x\in E_X$ such that $|x^*(x)|=1.$ In section 3, we use this structure of $E_{\mathcal{L}(X)_w^*},$ to explore best approximation and deduce distance formula in $\mathcal{L}(X)_w.$ For $T\in \mathcal{L}(X)_w,$ and a subspace $\mathcal{V}$ of $\mathcal{L}(X)_w,$ we characterize $T\perp_w\mathcal{V}.$  Then we obtain $d_w(T,\mathcal{V}),$ where $T\in \mathcal{L}(X)_w\setminus \mathcal{V}.$ A special attention is given to the case where $\mathscr{L}_{\mathcal{V}}(T)_w$ contains an operator $S$ such that $T-S$ is nu-smooth. Finally, we obtain an equivalence between B-J orthogonality in $\mathcal{L}(X)_w$ and that in $X.$ We show that for a subspace $Z$ of $X,$ if $T\in \mathcal{L}(X)_w$ is nu-smooth, then $T\perp_w \mathcal{L}(X,Z)$ if and only if $x_0\perp_B Z$ for some $x_0\in E_X,$ satisfying some other properties. Some results of section 3 generalize recently obtained results of \cite{MP,MPS,RS21}. Moreover, all the results of this section highlight the significance of the extremal structure of $B_{\mathcal{L}(X)_w^*}$ obtained in  section 2.  


\section{Extreme points of $B_{\mathcal{L}(X)_w^*}$}
We begin this section with the study of extreme points of $	B_{\mathcal{L}(X)_w^*},$ where $X$ is an arbitrary Banach space. 
\begin{theorem}\label{th-exin}
	Let $X$ be a Banach space. Then $	E_{\mathcal{L}(X)_w^*}\subseteq {\overline{\mathcal{A}}}^{w*},$ where
	\begin{equation}\label{eq-seta}
	\mathcal{A}=\{x^*\otimes x:x^*\in B_{X^*},x\in B_{X},|x^*(x)|=\|x^*\|\|x\|\}.
	\end{equation} 
\end{theorem}
\begin{proof}
	Observe that, ${}^o\mathcal{A}=B_{\mathcal{L}(X)_w}.$ Indeed 
	\begin{eqnarray*}
		{}^o\mathcal{A}&=& \Big\{T\in \mathcal{L}(X)_w:|x^*\otimes x(T)|\leq 1 ~\forall~ x^*\otimes x\in \mathcal{A}\Big\}\\
        &=& \Big\{T\in \mathcal{L}(X)_w:\sup_{x^*\otimes x\in \mathcal{A}}|x^*(Tx)|\leq 1 \Big\}\\
        &=& \Big\{T\in \mathcal{L}(X)_w:w(T)\leq 1 \Big\}\\
        &=& B_{\mathcal{L}(X)_w}.
	\end{eqnarray*}
Now,
\begin{eqnarray*}
({}^o\mathcal{A})^o&=&(B_{\mathcal{L}(X)_w})^o\\
&=&\{f\in {\mathcal{L}(X)_w}^*:|f(T)|\leq 1~\forall~f\in B_{\mathcal{L}(X)_w}\}\\
&=&\{f\in {\mathcal{L}(X)_w}^*:\|f\|\leq 1\}\\
&=&B_{\mathcal{L}(X)_w^*}.
\end{eqnarray*} 
Note that $\mathcal{A}$ is a Balanced set, i.e., for any scalar $\lambda$ with $|\lambda|\leq1,$ $\lambda \mathcal{A}\subseteq \mathcal{A}.$ Therefore, using a consequence of the Bipolar theorem (see \cite[Cor. 1.9, pp. 127]{CONW90}), we get,
\begin{equation}\label{eq-bipolar}
B_{\mathcal{L}(X)_w^*}=({}^o\mathcal{A})^o=\overline{co}^{w*}(\mathcal{A}).	
\end{equation}
 Clearly, $B_{\mathcal{L}(X)_w^*}$ is a weak*compact, convex subset of $\mathcal{L}(X)_w^*$ and $\mathcal{A}\subseteq B_{\mathcal{L}(X)_w^*}.$ Now, it follows from \cite[Th. 7.8, pp. 143]{CONW90}, that $E_{\mathcal{L}(X)_w^*}\subseteq {\overline{\mathcal{A}}}^{w*},$ completing the proof of the theorem. 
\end{proof}
In case, $X$ is considered to be finite-dimensional, the previous theorem can be strengthened further. We use the following simple lemma to get the desired result.
\begin{lemma}\label{lem-gen}
	Let $X$ be a Banach space. Let $x_1^*,x_2^*\in X^*,x_1,x_2\in X.$ Then the following are true.\\
	\rm(i) If  $x_1^*(x_1)\neq 0$ and $x_1^*\otimes x_1=x_1^*\otimes x_2,$ then $x_1=x_2.$\\
	 \rm(ii) If $x_1\neq 0$ and  $x_1^*\otimes x_1=x_2^*\otimes x_1,$ then $x_1^*=x_2^*.$
\end{lemma}
 \begin{proof}
 (i)	If possible, let $x_1\neq x_2.$ Then there exists $x^*\in X^*$ such that $x^*(x_1)\neq x^*(x_2).$ Define $T:X\to X$ by $Tx=x^*(x)x_1$ for all $x\in X.$ Then 
 	\begin{eqnarray*}
 		x_1^*\otimes x_1=x_1^*\otimes x_2 &\Rightarrow & x_1^*\otimes x_1(T)=x_1^*\otimes x_2(T)\\
 		&\Rightarrow &x_1^*(Tx_1)=x_1^*(Tx_2)	\\
 		&\Rightarrow& x_1^*(x_1)x^*(x_1)=x_1^*(x_1)x^*(x_2)\\
 		&\Rightarrow& x^*(x_1)=x^*(x_2),~\mbox{a contradiction.}
 	\end{eqnarray*}
 	Therefore, we must have $x_1=x_2.$\\
 (ii) Choose $x^*\in X^*$ such that $x^*(x_1)\neq 0.$ Let $u\in X$ be arbitrary. Define $T:X\to X$ by $Tx=x^*(x)u$ for all $x\in X.$ Then 
 \begin{eqnarray*}
 	x_1^*\otimes x_1=x_2^*\otimes x_1 &\Rightarrow& 	x_1^*\otimes x_1(T)=x_2^*\otimes x_1(T)\\
 &\Rightarrow&	x_1^*(Tx_1)=x_2^*(Tx_1)\\
 &\Rightarrow& x_1^*(u) x^*(x_1)=x_2^*(u)x^*(x_1)\\
 &\Rightarrow& x_1^*(u)=x_2^*(u).
  \end{eqnarray*}
Since $x_1^*(u)=x_2^*(u)$ for arbitrary $u\in X,$ we have $x_1^*=x_2^*.$
 \end{proof}
Now, we are ready to prove the main result of this section. 
\begin{theorem}\label{th-exfi}
	Let $X$ be a finite-dimensional Banach space. Then 
	\begin{equation}
		E_{\mathcal{L}(X)_w^*}\subseteq\{x^*\otimes x:x^*\in E_{X^*},x\in E_{X},|x^*(x)|=1\}.
	\end{equation} 
\end{theorem}
\begin{proof}
	Suppose  $\mathcal{A}$ is defined as in (\ref{eq-seta}).  First we show that $\mathcal{A}$ is a compact set. Consider the map 
	\[\Psi:X^*\times X\to \mathcal{L}(X)_w^*\] defined as $\Psi(x^*,x)=x^*\otimes x$ for all $x^*\in X^*,x\in X,$ i.e., for any $T\in \mathcal{L}(X)_w,$
	\[\Psi(x^*,x)(T)=x^*\otimes x(T)=x^*(Tx).\]
	Note that, $B_{X^*}\times B_X$ is compact in the product topology of $X^*\times X.$ Now, if the set $U=\{(x^*,x)\in B_{X^*}\times B_X: |x^*(x)|=\|x^*\|\|x\|\}$ is a closed subset of $B_{X^*}\times B_X$ and $\Psi$ is continuous, then $\Psi(U)=\mathcal{A}$ will be compact. To show $\Psi$ is continuous, let $\{(x_\alpha^*,x_\alpha)\}$ be a net in $X^*\times X$ such that $(x_\alpha^*,x_\alpha)\to (x^*,x).$ Then $x^*_\alpha\to x^*$ and $x_\alpha \to x.$ Now, for any $T\in \mathcal{L}(X)_w,$
	\begin{eqnarray*}
		&&|\Psi(x_\alpha^*,x_\alpha)(T)-\Psi(x^*,x)(T)|\\
		&=&|x_\alpha^*(Tx_\alpha)-x^*(Tx)|\\
		&\leq& |x_\alpha^*(Tx_\alpha)-x_\alpha^*(Tx)|+|x_\alpha^*(Tx)-x^*(Tx)|\\
		&\leq& \|x_\alpha^*\|\|Tx_\alpha-Tx\|+|x_\alpha^*(Tx)-x^*(Tx)|\\
		&&\longrightarrow 0.
	\end{eqnarray*}
Thus, $\{\Psi(x_\alpha^*,x_\alpha)\}$ weak*converges to $\Psi(x^*,x).$ Since weak*convergence and norm convergence in $\mathcal{L}(X)_w^*$ are same, $\Psi(x_\alpha^*,x_\alpha)\to \Psi(x^*,x).$ This proves that $\Psi$ is continuous. To prove $U$ is closed, assume that $\{(x_\alpha^*,x_\alpha)\}$ is a net in $U$ such that $(x_\alpha^*,x_\alpha)\to (x^*,x).$ Then $|x_\alpha^*(x_\alpha)|=\|x_\alpha^*\|\|x_\alpha\|,$ $x_\alpha^*\to x^*$ and $x_\alpha\to x.$ Hence $(x^*,x)\in B_{X^*}\times B_X$ and $|x_\alpha^*(x_\alpha)|\to \|x^*\|\|x\|.$ On the other hand,
\begin{eqnarray*}
	&&||x_\alpha^*(x_\alpha)|-|x^*(x)||\\
	&\leq& |x_\alpha^*(x_\alpha)-x^*(x)|\\
	&\leq & |x_\alpha^*(x_\alpha)-x^*(x_\alpha)|+|x^*(x_\alpha)-x^*(x)|\\
	&\leq& \|x_\alpha^*-x^*\|\|x_\alpha\|+\|x^*\|\|x_\alpha-x\|\\
	&&\longrightarrow 0.
\end{eqnarray*} 
Therefore, $|x_\alpha^*(x_\alpha)|\to |x^*(x)|.$ Thus, $|x^*(x)|=\|x^*\|\|x\|.$ This proves that $(x^*,x)\in U$ and $U$ is closed. Now, from the compactness of $B_{X^*}\times B_X$ and the fact that $U\subseteq B_{X^*}\times B_X,$ it follows that $U$ is compact. Therefore, $\Psi(U)=\mathcal{A}$ is compact. Using Theorem \ref{th-exin}, we get $$E_{\mathcal{L}(X)_w^*}\subseteq {\overline{\mathcal{A}}}^{w*}=\overline{\mathcal{A}}=\mathcal{A}.$$
Now, assume that $f\in E_{\mathcal{L}(X)_w^*}.$ Then $f=x^*\otimes x$ for some $x^*\in B_{X^*},x\in B_X$ with $|x^*(x)|=\|x^*\|\|x\|.$ To complete the proof of the theorem, we only have to show that $x^*\in E_{X^*}$ and $x\in E_{X}.$ First note that $x^*\in S_{X^*}$ and $x\in S_{X}.$ For otherwise, choose $T\in S_{\mathcal{L}(X)_w}$ such that $f(T)=w(T)=1,$ i.e., $x^*\otimes x(T)=w(T).$ Then $\Big|\frac{x^*}{\|x^*\|}(T\frac{x}{\|x\|})\Big|=\frac{w(T)}{\|x^*\|\|x\|}>w(T),$ a contradiction. Therefore, $x^*\in S_{X^*}$ and $x\in S_{X}.$ Now, if possible, suppose that $x\notin E_{X}.$ Then there exist $t\in (0,1), x_1,x_2\in B_{X}$ such that $x_1\neq x\neq x_2$ and $x=tx_1+(1-t)x_2.$ Thus, from
$$1=|x^*(x)|=|tx^*(x_1)+(1-t)x^*(x_2)|\leq t|x^*(x_1)|+(1-t)|x^*(x_2)|\leq 1,$$
we get $|x^*(x_1)|=|x^*(x_2)|=1$ and $\|x_1\|=\|x_2\|=1.$ Therefore, $x^*\otimes x_1,x^*\otimes x_2\in \mathcal{A}\subseteq B_{\mathcal{L}(X)_w^*}.$ Moreover, $x^*\otimes x=tx^*\otimes x_1+(1-t)x^*\otimes x_2.$ Thus, from $x^*\otimes x\in E_{\mathcal{L}(X)_w^*},$ it follows that $x^*\otimes x=x^*\otimes x_1=x^*\otimes x_2.$ Therefore, by Lemma \ref{lem-gen}, we have $x=x_1=x_2$ and thus $x\in E_X.$ Similarly, it can be shown that $x^*\in E_{X^*}.$ This completes the proof of the theorem.
\end{proof}

\section{Best approximation and distance formula in $\mathcal{L}(X)_w$}
In this section, we  apply the results of the previous section to study best approximation and obtain distance formula in $\mathcal{L}(X)_w.$ In this direction, we first characterize $T\perp_w \mathcal{V},$ where $X$ is a finite-dimensional Banach space,  $T\in \mathcal{L}(X)_w$ and $\mathcal{V}$ is a subspace of $\mathcal{L}(X)_w.$
\begin{theorem}\label{th-nubjfi}
	Let $X$ be a finite-dimensional Banach space. Let $\mathcal{V}$ be a $n$-dimensional subspace of $\mathcal{L}(X)_w$ and $T\in \mathcal{L}(X)_w.$ Then  the following are equivalent.\\
	
\noindent	\rm(i) $T\perp_w \mathcal{V}$

\noindent	\rm(ii) There exist $t_1,t_2,\ldots,t_h>0$ ($h\leq n+1$ if $\mathbb{F}=\mathbb{R}$ and $h\leq 2n+1$ if $\mathbb{F}=\mathbb{C}$), $x_i^*\in E_{X^*},x_i\in E_X$ for all $1\leq i\leq h$ such that 
$$|x_i^*(x_i)|=1,~x_i^*(Tx_i)=w(T),~\sum_{i=1}^h t_i=1, \mbox{ and } \sum_{i=1}^ht_ix_i^*(Sx_i)=0~\mbox{for all } S\in \mathcal{V}.$$

\noindent	\rm(iii) $0\in co(D_S )$ for all $S\in \mathcal{V},$ where
	\[D_S=\{x^*(Sx):x^*\in E_{X^*},x\in E_x,|x^*(x)|=1,x^*(Tx)=w(T)\}.\]
\end{theorem}
\begin{proof}
(i) $\Rightarrow$ (ii). Let $T\perp_w \mathcal{V}.$ Then using Theorem \ref{th-singer}, we get scalars $t_1,t_2,\ldots,t_h>0,$ where $h$ is as in (ii) and $f_1,f_2,\ldots,f_h\in E_{\mathcal{L}(X)_w^*}$ such that $\sum_{i=1}^h t_i=1,$ $f_i(T)=w(T)$ for all $1\leq i\leq h$ and  $\sum_{i=1}^h t_if_i(S)=0$ for all $S\in \mathcal{V}.$  From Theorem \ref{th-exfi}, it follows that $f_i=x_i^*\otimes x_i,$ where $x_i^*\in E_{X^*},x_i\in E_X$ and $|x_i^*(x_i)|=1.$ Thus, 
	$$x_i^*(Tx_i)=x_i^*\otimes x_i(T)=f_i(T)=w(T)~\mbox{and}$$
	\[\sum_{i=1}^ht_ix_i^*(Sx_i)=\sum_{i=1}^ht_ix_i^*\otimes x_i(S)=\sum_{i=1}^h t_if_i(S)=0~\mbox{for all } S\in \mathcal{V}.\]
	This proves (ii).\\
(ii) $\Rightarrow$ (iii) is trivial.\\
(iii) $\Rightarrow$ (i). Let $S\in \mathcal{V}.$ Then $0\in co(D_S)$ implies that there exist scalars $t_1,t_2,\ldots, t_n>0$ such that $\sum_{i=1}^n t_i=1$ and $\sum_{i=1}^n t_ix_i^*(Sx_i)=0,$ where $x_i^*(Sx_i)\in D_S.$ Suppose that $f=\sum_{i=1}^n t_ix_i^*\otimes x_i\in \mathcal{L}(X)_w^*.$ Then $f(S)=0.$  
Note that, for any $A\in \mathcal{L}(X)_w,$	
$$f(A)=\sum_{i=1}^n t_ix_i^*(Ax_i)\leq \sum_{i=1}^n t_iw(A)=w(A).$$
Thus, $\|f\|\leq 1.$ Moreover, $f(T)=w(T)$ implies that $f\in J_w(T).$ Thus, from  $f(S)=0,$ it follows that $T\perp_w S.$ Since this is true for each $S\in \mathcal{V},$ we conclude that $T\perp_w\mathcal{V}.$	
\end{proof}
\begin{remark}
We would like to note that in \cite[Th. 2.3]{RS21}, the authors characterized $T\perp_w A,$ where $T,A\in \mathcal{L}(X)_w.$ They proved that $T\perp_wA$ if and only if $0\in co(D),$ where $D=\{\overline{x^*(Tx)}x^*(Ax):x\in S_X,x^*\in S_{X^*},x^*(x)=1,|x^*(Tx)|=w(T)\}.$ Clearly, the equivalence (i) $\Leftrightarrow $ (iii) of Theorem \ref{th-nubjfi}  improves \cite[Th. 2.3]{RS21} when $\mathcal{V}$ is one-dimensional and either  $E_X\subsetneq S_X$  or  $E_{X^*}\subsetneq S_{X^*},$ which is always true if $X$ is a finite-dimensional polyhedral Banach space.
\end{remark}
In \cite{MPS}, we characterized $T\perp_w A,$ where $T,A$ are defined over Hilbert spaces. Then in \cite[Th. 2.11]{MP}, we obtained a sufficient condition for $T\perp_w \mathcal{V},$ where $T$ is defined on a complex Hilbert space $H$ and  $\mathcal{V}$ is any subspace of $\mathcal{L}(H)_w.$ In the next theorem, we prove that the sufficient condition is also necessary if $H$ is finite-dimensional.
\begin{theorem}
	Let $H$ be a finite-dimensional complex Hilbert space. Suppose that $(0\neq) T\in \mathcal{L}(H)_w$ and $\mathcal{V}$ is a $n$-dimensional subspace of $\mathcal{L}(H)_w.$ Then the following are equivalent.\\
	\rm(i) $T\perp_w \mathcal{V}.$\\
	\rm(ii) There exist $t_1,t_2,\ldots,t_h>0$ ($h\leq 2n+1$), $x_i\in S_H$ for all $1\leq i\leq h$ such that $|\langle Tx_i,x_i\rangle|=w(T)$ for all $1\leq i\leq h,$ $\sum_{i=1}^n t_i=1$ and 
	\[\sum_{i=1}^n t_i \overline{\langle Tx_i,x_i\rangle}\langle Sx_i,x_i\rangle=0~\mbox{for all } S\in\mathcal{V}.\] 
\end{theorem}
\begin{proof}
(ii) $\Rightarrow$ (i) follows from \cite[Th. 2.11]{MP}. \\
 (i) $\Rightarrow$ (ii). Let $T\perp_w \mathcal{V}.$ Then from Theorem \ref{th-nubjfi}, we get scalars $t_i>0,$ $x_i^*\in S_{H^*},x_i\in S_H$ for $1\leq i\leq h$ ($h\leq 2n+1$) such that  
 $$|x_i^*(x_i)|=1,~x_i^*(Tx_i)=w(T),~\sum_{i=1}^h t_i=1, \mbox{ and } \sum_{i=1}^ht_ix_i^*(Sx_i)=0~\mbox{for all } S\in \mathcal{V}.$$
 Suppose that $x_i^*(x_i)=\mu_i,$ i.e., $|\mu_i|=1$ and $\overline{\mu_i} x_i^*(x_i)=1.$ Thus, $\overline{\mu_i}x_i^*\in J(x_i).$ Note that, $H$ is smooth, i.e., $J(x_i)=\{\overline{\mu_i}x_i^*\}.$  Therefore, $\overline{\mu_i}x_i^*(x)=\langle x,x_i\rangle\Rightarrow x_i^*(x)=\mu_i \langle x,x_i\rangle$ for all $x\in H.$	Now,
 \[\overline{\langle Tx_i,x_i\rangle}=\mu_ix_i^*(Tx_i)=\mu_iw(T).\]
 Thus, $|\langle Tx_i,x_i\rangle|=w(T)$ and for all $S\in\mathcal{V},$
 \begin{eqnarray*}
 	\sum_{i=1}^ht_ix_i^*(Sx_i)=0
 	&\Rightarrow&\sum_{i=1}^ht_i \mu_i \langle Sx_i, x_i\rangle=0\\
 	&\Rightarrow&\sum_{i=1}^ht_i \mu_i w(T) \langle Sx_i, x_i\rangle=0\\
 	&\Rightarrow&\sum_{i=1}^ht_i \overline{\langle Tx_i,x_i\rangle} \langle Sx_i, x_i\rangle=0.
 \end{eqnarray*}
This proves (ii).
\end{proof}
In the next theorem, we generalize Theorem \ref{th-nubjfi}.  In particular, we characterize $T\perp_w \mathcal{V},$ for $T\in \mathcal{L}(X)_w$ and a subspace $\mathcal{V}$ of $\mathcal{L}(X)_w,$ whenever $X$ is an arbitrary Banach space, not necessarily finite-dimensional.
\begin{theorem}\label{th-nubjin}
	Let $X$ be a Banach space. Let $\mathcal{V}$ be a $n$-dimensional subspace of $\mathcal{L}(X)_w$ and $T\in \mathcal{L}(X)_w.$ Then $T\perp_w \mathcal{V}$ if and only if the following are true.\\
	
	\noindent	\rm(i) There exist scalars $t_1,t_2,\ldots,t_h>0$ ($h\leq n+1$ if $\mathbb{F}=\mathbb{R}$ and $h\leq 2n+1$ if $\mathbb{F}=\mathbb{C}$) such that $\sum_{i=1}^h t_i=1.$
	
	\noindent	\rm(ii)  For each $1\leq i\leq h,$ there exists a net $\{x_{i\alpha}^*\otimes x_{i\alpha}\}_\alpha\in \mathcal{A}$ such that $$\lim_{\alpha}x^*_{i\alpha}(Tx_{i\alpha})=w(T)~\text{and}~\sum_{i=1}^h t_i\lim_{\alpha}x^*_{i\alpha}(Sx_{i\alpha})=0~\mbox{for all } S\in \mathcal{V}.$$
\end{theorem}
\begin{proof}
	First suppose that $T\perp_w \mathcal{V}.$ Then by Theorem \ref{th-singer}, we get  $t_1,t_2,\ldots,t_h>0$ satisfying (i)  and $f_1,f_2,\ldots,f_h\in E_{\mathcal{L}(X)_w^*}$ such that $f_i(T)=w(T)$ for all $1\leq i\leq h$ and $\sum_{i=1}^h t_if_i(S)=0$ for all  $S\in \mathcal{V}.$  Now, from Theorem \ref{th-exin}, it follows that $f_i\in \overline{\mathcal{A}}^{w*}.$ Therefore, there exists a net  $\{x_{i\alpha}^*\otimes x_{i\alpha}\}_\alpha\in \mathcal{A}$ such that the net $\{x_{i\alpha}^*\otimes x_{i\alpha}\}_\alpha$ is weak*convergent to $f_i.$ (Note that we can choose same index set for the nets for all  $1\leq i\leq h$). Hence,  $$\lim_{\alpha}x_{i\alpha}^*(Tx_{i\alpha})=f_i(T)=w(T) ~\text{and}$$  $$\lim_{\alpha}x_{i\alpha}^*(Sx_{i\alpha})=f_i(S)\Rightarrow \sum_{i=1}^h t_i\lim_{\alpha}x^*_{i\alpha}(Sx_{i\alpha})=\sum_{i=1}^h t_if_i(S)=0~\mbox{for all } S\in \mathcal{V} .$$
	Conversely, assume that (i) and (ii) are true. Then for any $S\in \mathcal{V},$
	\begin{eqnarray*}
			w(T+S)&\geq& |x_{i\alpha}^*\otimes x_{i\alpha}(T+S)|\\
		&=& |x_{i\alpha}^*((T+S) x_{i\alpha})|\\
		&\geq & \Re\Big(x_{i\alpha}^*((T+S) x_{i\alpha})\Big)\\
		\Rightarrow w(T+S)&\geq & \lim_{\alpha} \Re\Big(x_{i\alpha}^*(T x_{i\alpha})\Big)+\lim_{\alpha} \Re\Big(x_{i\alpha}^*(S x_{i\alpha})\Big)\\  
		\Rightarrow \sum_{i=1}^h t_i w(T+S)&\geq & \sum_{i=1}^h t_i\lim_{\alpha} \Re\Big(x_{i\alpha}^*(T x_{i\alpha})\Big)+\sum_{i=1}^h t_i\lim_{\alpha} \Re\Big(x_{i\alpha}^*(S x_{i\alpha})\Big)\\ 
		\Rightarrow w(T+S)&\geq & \sum_{i=1}^h t_i\lim_{\alpha} \Re\Big(x_{i\alpha}^*(T x_{i\alpha})\Big)\\
		\Rightarrow w(T+S)&\geq & \sum_{i=1}^h t_i w(T)=w(T).
	\end{eqnarray*}	
Thus, $T\perp_w S$ for all $S\in \mathcal{V},$ which proves that $T\perp_w\mathcal{V}.$ 
\end{proof}		
As an immediate application of Theorem \ref{th-nubjfi}, we obtain the following distance formula in $\mathcal{L}(X)_w.$
\begin{theorem}\label{th-dis}
	Let $X$ be a finite-dimensional Banach space. Let $\mathcal{V}$ be a $n$-dimensional subspace of $\mathcal{L}(X)_w$ and $T\in \mathcal{L}(X)_w\setminus \mathcal{V}.$ Then there exist scalars $t_1,t_2,\ldots,t_h>0,$ ($h\leq n+1$ if $\mathbb{F}=\mathbb{R}$ and $h\leq 2n+1$ if $\mathbb{F}=\mathbb{C}$) such that $\sum_{i=1}^ht_i=1$ and  
	\begin{eqnarray*}
		& & d_w(T,\mathcal{V})  \\
		&=& \max\Big\{\sum_{i=1}^ht_ix_i^{*}(Tx_i):x_i^*\in E_{\mathbb{X}^*},~  x_i\in E_{\mathbb{X}},~|x_i^*(x_i)|=1,\sum_{i=1}^ht_i\Im(x_i^*(Tx_i))=0,\\
		&& \mbox{ \hspace{7.3cm}} \sum_{i=1}^ht_ix_i^{*}(Sx_i)=0 ~\forall~ S\in \mathcal{V}\Big\}.
	\end{eqnarray*} 
\end{theorem}
\begin{proof}
	From \cite[Th. 2.1]{MP22} it follows that there exist scalars $t_1,\ldots,t_h$ as stated in the theorem such that 
		\begin{eqnarray*}
		& & d_w(T,\mathcal{V})  \\
		&=& \max\Big\{\sum_{i=1}^ht_ig_i(T): g_i\in E_{\mathcal{L}(X)_w^*},\sum_{i=1}^ht_i\Im(g_i(T))=0, \sum_{i=1}^ht_ig_i(S)=0 ~\forall~ S\in \mathcal{V}\Big\}.
	\end{eqnarray*}
Since $g_i\in E_{\mathcal{L}(X)_w^*},$ using Theorem \ref{th-exfi}, we get $g_i=x_i^*\otimes x_i,$ where $x_i^*\in E_{X^*},x_i\in E_X$ and $|x_i^*(x_i)|=1.$ Now, the proof follows simply observing that 
\[g_i(T)=x_i^*\otimes x_i(T)=x_i^*(Tx_i)~\text{and}~g_i(S)=x_i^*\otimes x_i(S)=x_i^*(Sx_i).\]
\end{proof}
In particular, if $X$ is finite-dimensional real Banach space and $\mathcal{V}$ is one-dimensional, then Theorem \ref{th-dis} assumes the following form.  
\begin{cor}
		Let $X$ be a finite-dimensional real Banach space. Let  $T,S\in \mathcal{L}(X)_w$ such that $T\notin span\{S\}.$ Then 
			\begin{eqnarray*}
			& & d_w(T,span\{S\})  \\
			& =& \max\Big\{t x_1^*(Tx_1)+(1-t)x_2^*(Tx_2) ~{\textbf :}  ~t\in[0,1],  x_i\in E_{\mathbb{X}}, x_i^*\in E_{\mathbb{X}^*},|x_i^*(x_i)|=1,\\
			&~& \mbox{\hspace{5.5cm}}i=1,2, tx_1^*(Sx_1)+(1-t)x_2^*(Sx_2)=0\Big\}.
		\end{eqnarray*}
\end{cor}
In the next part of this section, we give special attention to the nu-smooth operators. We show that if $S$ is a best approximation to $T$ out of a subspace $\mathcal{V}$ such that $T-S$ is a nu-smooth operator, then the distance formula $d_w(T,\mathcal{V})$ assumes a simple form compared to Theorem \ref{th-dis}. Before that, let us first characterize nu-smooth operators. The characterization of nu-smooth operators depends on the numerical radius attainment set of $T,$ denoted as $M_{w(T)}$ and defined as  
\[M_{w(T)}=\{(x,x^*)\in S_X\times S_{X^*}:x^*(x)=1,|x^*(Tx)|=w(T)\}.\]

\begin{theorem}\label{th-smooth}
		Let $X$ be a finite-dimensional  Banach space. Let  $T\in \mathcal{L}(X)_w.$ Then the following are equivalent.\\
		\rm(i) $T$ is nu-smooth.\\
		\rm(ii) $M_{w(T)}=\{(\mu x_0,\overline{\mu} x_0^*):|\mu|=1\}$ for some $x_0^*\in E_{X^*},x_0\in E_X$ with $x_0^*(x_0)=1.$
\end{theorem}
\begin{proof}
	(i) $\Rightarrow$ (ii). Let $T$ be nu-smooth. Then $J_w(T)=\{f\}$ for some $f\in S_{\mathcal{L}(X)_w^*}.$ Clearly, $f$ is an extreme point of $J_w(T).$  Since $J_w(T)$ is an extremal subset of $B_{\mathcal{L}(X)_w^*},$ $f$ must be an extreme point of $B_{\mathcal{L}(X)_w^*}.$ Thus, by Theorem \ref{th-exfi}, $f=x_0^*\otimes x_0,$ where $x_0^*\in E_{X^*}, x_0\in E_X$ and $|x_0^*(x_0)|=1,$ i.e., $x_0^*(x_0)=\alpha$ for some $|\alpha|=1.$ Thus, $\overline{\alpha} x_0^*(x_0)=1.$ Now, 
	$$|\overline{\alpha}f(T)|=w(T)\Rightarrow |\overline{\alpha}x_0^*\otimes x_0(T)|=w(T)\Rightarrow |\overline{\alpha}x_0^*(Tx_0)|=w(T).$$
	Thus, $(x_0,\overline{\alpha} x_0^*)\in M_{w(T)},$ i.e., $(\mu x_0,\overline{\mu}\overline {\alpha} x_0^*)\in M_{w(T)}$ for all scalars $\mu$ with $|\mu|=1.$ Now, suppose that $(x,x^*)\in M_{w(T)}.$ Then $x^*(x)=1$ and $|x^*(Tx)|=w(T),$ i.e., there exists a scalar $\mu$ with $|\mu|=1$ such that $\mu x^*\otimes x(T)=w(T).$ Since $\mu x^*\otimes x\in B_{\mathcal{L}(X)_w^*},$ we must have $\|\mu x^*\otimes x\|=1$ and $\mu x^*\otimes x\in J_w(T)=\{x_0^*\otimes x_0\}.$ Thus, we get
	 $$\mu x^*\otimes x=x_0^*\otimes x_0.$$
	   Now, from \cite[Lem. 3.1]{MDP} it follows that either $\{\mu x^*,x_0^*\}$ or $\{x,x_0\}$ is linearly dependent. Note that, if $\mu x^*=\lambda x_0^*$ for some scalar $\lambda,$ then 
	 $$\mu x^*\otimes x= x_0^*\otimes x_0\Rightarrow \lambda x_0^*\otimes x=x_0^*\otimes x_0\Rightarrow x_0^*\otimes \lambda x=x_0^*\otimes x_0.$$ Therefore, by Lemma \ref{lem-gen}, we get $ \lambda x=x_0.$ Now, $x^*(x)=1\Rightarrow x_0^*( x_0)=\mu,$ which shows that $\mu=\alpha.$ Thus, $(x,x^*)=(\frac{1}{ \lambda}x_0,\frac{\lambda}{\alpha} x_0^*)=(\overline{\lambda}x_0,\lambda \overline{\alpha}x_0^*).$ On the other hand, if $x=\rho x_0$ for some scalar $\rho,$ then $|\rho|=1$ and
	  $$\mu x^*\otimes x=x_0^*\otimes x_0\Rightarrow \mu x^*\otimes \rho x_0= x_0^*\otimes x_0\Rightarrow \mu \rho x^*\otimes x_0=x_0^*\otimes x_0.$$ Thus, again by Lemma \ref{lem-gen}, we have $x_0^*=\mu \rho x^*.$ Note that $x^*(x)=1$ gives $x_0^*(x_0)=\mu,$ i.e., $\mu=\alpha.$ Thus, $(x,x^*)=(\rho x_0,\frac{1}{\alpha \rho}x_0^*)=(\rho x_0,\overline{\rho }\overline{\alpha }x_0^*).$ Hence, we get
	  \[M_{w(T)}=\{(\mu x_0,\overline {\mu} \overline{\alpha} x_0^*):|\mu|=1\},\] where $\overline {\alpha} x_0^*(x_0)=1,$ $\overline{\alpha}x_0^*\in E_{X^*},x\in E_X.$ This proves (ii).\\
	  
	  (ii) $\Rightarrow$ (i). If possible, suppose $T$ is not nu-smooth, i.e., $J_w(T)$ is not singleton. Then $J_w(T)$ contains at least two distinct extreme points. Let $f_1,f_2$ be two distinct extreme points of $J_w(T).$ Since  $J_w(T)$ is an extremal subset of $B_{\mathcal{L}(X)_w^*},$ $f_1,f_2\in E_{\mathcal{L}(X)_w^*}.$ Thus, from Theorem \ref{th-exfi}, we get for $i=1,2,$ $f_i=x_i^*\otimes x_i$ for some $x_i^*\in E_{X^*},x_i\in E_X$ with $|x_i^*(x_i)|=1.$  Let $x_i^*(x_i)=\alpha_i,$ where $|\alpha_i|=1.$ Then as before, we get $(x_i,\overline{\alpha_i} x_i^*)\in M_{w(T)}.$ Now, from (ii), we have $(x_i,\overline{\alpha_i} x_i^*)=(\mu_i x_0,\overline{\mu_i} x_0^*)$ for $i=1,2,$ where $|\mu_i|=1.$  Thus, $x_i=\mu_i x_0$ and $ \overline{\alpha_i}x_i^*=\overline{\mu_i}x_0^*.$ This gives that $x_i^*\otimes x_i=\alpha_i x_0^*\otimes x_0.$ Now, from 
	  \[w(T)=x_i^*\otimes x_i(T)=\alpha_i x_0^*\otimes x_0(T)=\alpha_ix_0^*(Tx_0),\]
	  it follows that $\alpha_1=\alpha_2,$ i.e., $x_1^*\otimes x_1=x_2^*\otimes x_2.$ This contradicts that $f_1,f_2$ are distinct. Thus $T$ is nu-smooth.
\end{proof}
We would like to mention that Theorem \ref{th-smooth} was first proved in \cite[Th. 2.5]{RS21} with a completely different approach. The proof presented here once again indicates the importance of Theorem \ref{th-exfi}. 
Next, we obtain the desired distance formula in  $\mathcal{L}(X)_w.$ 

\begin{theorem}\label{th-smoothdis}
	Let $\mathbb{X}$ be a finite-dimensional Banach space. Let $T\in \mathcal{L}(X)_w$ and  $\mathcal{V}$ be a subspace of $\mathcal{L}(X)_w.$ Suppose there exists $S\in \mathscr{L}_{\mathcal{V}}(T)_w$ such that $T-S$ is nu-smooth. Then 
	\begin{eqnarray*}
		&& d_w(T,\mathcal{V})\\
		&=&\max\Big\{x^*(Tx):x\in E_{X},x^*\in E_{X^*}, |x^*(x)|=1, \Im(x^*(Tx))=0,\\ && \hspace{6.2cm}x^*(Ax)=0~\forall~A\in\mathcal{V}\Big\}.
	\end{eqnarray*}
\end{theorem}
\begin{proof}
	Assume that $T-S=T_0.$ Then from $S\in \mathscr{L}_{\mathcal{V}}(T)_w,$ we get $T-S\perp_w \mathcal{V},$ i.e., $T_0\perp_w \mathcal{V}.$ Since $T_0$ is nu-smooth, proceeding similarly as Theorem \ref{th-smooth}, we get $J_w(T_0)=\{x_0^*\otimes x_0\}$ for some $x_0\in E_X,x_0^*\in E_{X^*},|x_0^*(x_0)|=1.$ Thus, for all $A\in\mathcal{V},$ $x_0^*\otimes x_0(A)=0\Rightarrow x_0^*(Ax_0)=0.$ Now,
	 $$x_0^*(Tx_0)=x_0^*\otimes x_0(T)= x_0^*\otimes x_0(T_0)=w(T_0)=d_w(T,\mathcal{V}).$$
	 This clearly implies that $\Im(x_0^*(Tx_0))=0.$ Hence,
	 	\begin{eqnarray*}
	 	 d_w(T,\mathcal{V})&=&x_0^*(Tx_0)\\
	 	&\leq&\sup\Big\{x^*(Tx):x\in E_{X},x^*\in E_{X^*}, |x^*(x)|=1, \Im(x^*(Tx))=0,\\ && \hspace{6.2cm}x^*(Ax)=0~\forall~A\in\mathcal{V}\Big\}.
	 \end{eqnarray*}
 On the other hand, observe that if $x\in E_{X},x^*\in E_{X^*}$ with $|x^*(x)|=1$ satisfy $ \Im(x^*(Tx))=0$ and $x^*(Ax)=0$ for all $A\in\mathcal{V},$ then 
 \[x^*(Tx)=\Re(x^*(Tx))=\Re(x^*(Tx-Ax))\leq|x^*(Tx-Ax)|\leq w(T-A).\]
 Thus, $x^*(Tx)\leq \inf\{w(T-A):A\in \mathcal{V}\}=d_w(T,\mathcal{V}).$ Therefore, we get
 	\begin{eqnarray*}
 &&\sup\Big\{x^*(Tx):x\in E_{X},x^*\in E_{X^*}, |x^*(x)|=1, \Im(x^*(Tx))=0,\\ && \hspace{6.2cm}x^*(Ax)=0~\forall~A\in\mathcal{V}\Big\}\\
 	&\leq& d_w(T,\mathcal{V}).	
 \end{eqnarray*}
This proves that the above inequality is actually an equality. Moreover, the supremum is attained at $x_0^*(Tx_0).$ This completes the proof of the theorem.  
\end{proof}
We end this article with a nice equivalence between B-J orthogonality  in $\mathcal{L}(X)_w$ and that in $X.$
\begin{theorem}
	Let $X$ be a finite-dimensional Banach space and $Z$ be a subspace of $X.$ Let  $T\in \mathcal{L}(X)_w$ be nu-smooth. Suppose that  $M_{w(T)}=\{(\mu x_0,\overline{\mu} x_0^*):|\mu|=1\}$ for some $x_0^*\in E_{X^*},x_0\in E_X$ with $x_0^*(x_0)=1.$ Assume that $x_0$ is smooth in $X.$  Then 
	$$T\perp_w\mathcal{L}(X,Z) \Leftrightarrow x_0\perp_B Z.$$ 
\end{theorem}
\begin{proof}
	From $(x_0,x_0^*)\in M_{w(T)},$ it follows that $|x_0^*(Tx_0)|=w(T)$ and $x_0^*(x_0)=1.$ Let $\mu x_0^*(Tx_0)=w(T)$ for some scalar $\mu$ with $|\mu|=1.$ Then $\mu x_0^*\otimes x_0(T)=w(T),$ where $|\mu x_0^*(x_0)|=1.$  Hence, $\mu x_0^*\otimes x_0\in J_w(T).$ Thus, $J_w(T)=\{\mu x_0^*\otimes x_0\},$ since $T$ is nu-smooth.\\
	First suppose that $T\perp_w\mathcal{L}(X,Z).$ Then $\mu x_0^*\otimes x_0(S)=0$ for all $S\in \mathcal{L}(X,Z).$ Choose arbitrary $z\in Z.$ Consider $S\in \mathcal{L}(X,Z)$ defined as $S(x)=x_0^*(x)z.$ Then
	\[\mu x_0^*\otimes x_0(S)=0\Rightarrow \mu x_0^*(Sx_0)=0\Rightarrow x_0^*(z)=0\Rightarrow x_0\perp_B z.\]
	Thus, we get $x_0\perp_B Z.$
	
	\noindent	Conversely, suppose that $x_0\perp_B Z.$ Since $x_0$ is smooth, $J(x_0)=\{x_0^*\}$ and so $x_0^*(z)=0$ for all $z\in Z.$ Since for any $S\in \mathcal{L}(X,Z),$ $Sx_0\in Z,$ we have $x_0^*(Sx_0)=0,$ i.e., $\mu x_0^*\otimes x_0(S)=0.$ Thus, $T\perp_w\mathcal{L}(X,Z).$ This completes the proof.
	
\end{proof}

\bibliographystyle{amsplain}

\end{document}